\documentclass[reqno]{amsart}
\usepackage[backend=bibtex, sorting=none, bibstyle=alphabetic, citestyle=alphabetic, sorting=nyt,maxbibnames=99]{biblatex}  
\renewbibmacro{in:}{}
\bibliography{final.bib}
\usepackage{amssymb}
\usepackage{calrsfs}
\usepackage{graphics,graphicx}
\usepackage{enumerate}
\usepackage{enumitem}
\usepackage{url}
\usepackage{xcolor}
\usepackage[ruled, lined, linesnumbered, longend]{algorithm2e}
\usepackage{hyperref}
\usepackage[justification=centering]{caption}

\newtheorem{theorem}{Theorem}[section]
\newtheorem{lemma}[theorem]{Lemma}

\numberwithin{equation}{section}

\theoremstyle{remark}
\newtheorem*{remark}{Remark}

\DeclareMathOperator{\li}{li}

\def\reals{\hbox{\rm I\kern-.18em R}}
\def\complexes{\hbox{\rm C\kern-.43em
\vrule depth 0ex height 1.4ex width .05em\kern.41em}}
\def\field{\hbox{\rm I\kern-.18em F}} 

\allowdisplaybreaks

\begin{document}

\title[On the average value of $\pi(t)-\li(t)$]{On the average value of $\pi(t)-\li(t)$}

\author{Daniel R. Johnston}
\address{School of Science, The University of New South Wales, Canberra, Australia}
\email{daniel.johnston@adfa.edu.au}
\date\today
\keywords{}

\begin{abstract}
    We prove that the Riemann hypothesis is equivalent to the condition $\int_{2}^x\left(\pi(t)-\li(t)\right)\mathrm{d}t<0$ for all $x>2$. Here, $\pi(t)$ is the prime-counting function and $\li(t)$ is the logarithmic integral. This makes explicit a claim of Pintz (1991). Moreover, we prove an analogous result for the Chebyshev function $\theta(t)$ and discuss the extent to which one can make related claims unconditionally.
\end{abstract}

\maketitle

\section{Introduction}
Let $\pi(x)$ denote the number of primes less than or equal to $x$ and 
\begin{equation*}
    \li(x)=\int_0^x\frac{1}{\log t}\mathrm{d}t.
\end{equation*}
In his celebrated 1859 article, Riemann \cite{Riemann_1859} remarked on the apparent truth of the inequality
\begin{equation*}
    \pi(x)<\li(x),
\end{equation*}
for all $x\geq 2$.
In 1903, Schmidt \cite[p. 204]{schmidt1903anzahl} showed that such a result would imply the Riemann hypothesis. However, in 1914 Littlewood \cite{littlewood1914distribution} managed to prove that $\pi(x)-\li(x)$ changes sign infinitely often. More precisely, he showed that for some positive constant $c$, there are arbitrarily large values of $x$ such that
\begin{equation*}
    \pi(x)-\li(x)>\frac{c\sqrt{x}\log\log\log x}{\log x}\quad\text{and}\quad\pi(x)-\li(x)<-\frac{c\sqrt{x}\log\log\log x}{\log x}.
\end{equation*}
It is an open problem to determine the smallest value of $x$ such that $\pi(x)>\li(x)$. Large computations have shown that $\pi(x)<\li(x)$ for all $x\leq 10^{19}$ \cite[Theorem 2]{buthe2018analytic} and that the first sign change occurs before $x=1.4\cdot 10^{316}$ \cite{bays2000new}, \cite{saouter2015still}.

Although $\pi(x)<\li(x)$ is not true in general, one can ask whether, in a precise sense, $\pi(x)<\li(x)$ is true on average. Namely, several authors \cite{ingham1932distribution}, \cite{kaczorowski1985sign}, \cite{pintz1991assertion}, \cite{stechkin1996asymptotic}, \cite{karatsuba2004riemann} assert that the Riemann hypothesis implies
\begin{equation}\label{averageeq}
    \int_{2}^x\left(\pi(t)-\li(t)\right)\mathrm{d}t<0,\quad x>x_0,
\end{equation}
for some sufficiently large $x_0$. Pintz \cite{pintz1991assertion} claims that \eqref{averageeq} is in fact equivalent to the Riemann hypothesis and is likely to hold for all $x>2$ under such assumptions. Using explicit bounds on prime counting functions we are able to prove this claim.

\begin{theorem}\label{rhthm1}
    The Riemann hypothesis is equivalent to the condition
    \begin{equation}
        \int_{2}^x\left(\pi(t)-\li(t)\right)\mathrm{d}t<0,\quad\text{for all}\ x>2.\label{piineq}
    \end{equation}
\end{theorem}

We also prove an analogous result for the Chebyshev function
\begin{equation*}
    \theta(x)=\sum_{p\leq x}\log p.
\end{equation*}

\begin{theorem}\label{rhthm2}
    The Riemann hypothesis is equivalent to the condition
    \begin{equation}
        \int_{2}^x\left(\theta(t)-t\right)\mathrm{d}t<0,\quad\text{for all}\ x>2.\label{thetaineq}
    \end{equation}
\end{theorem}

It is natural to ask whether a modification of \eqref{piineq} or \eqref{thetaineq} is true unconditionally. In this direction we consider the weighted integrals
\begin{equation}\label{weightedeq}
    I_1(x,f)=\int_2^x(\pi(t)-\li(t))f(t)\:\mathrm{d}t\quad\text{and}\quad I_2(x,f)=\int_2^x(\theta(t)-t)f(t)\:\mathrm{d}t
\end{equation}
for some choice of function $f(t)$. Given that one has to go quite far to find a value of $t$ such that\footnote{The first value of $t$ with $\theta(t)-t>0$ is also expected to be around $10^{316}$, see \cite{platt2016first}.} $\pi(t)-\li(t)>0$, one should intuitively take $f(t)$ to be positive and decreasing as to give more weight to the negative bias for small values of $t$. In \cite{pintz1991assertion}, Pintz considers $f(t)=\exp(-\log^2t/y)$ for sufficiently large $y$. However, using an explicit form of Mertens' theorems, we show that $I_1(x,f)<0$ and $I_2(x,f)<0$ for the simpler and asymptotically larger function $f(t)=1/t^2$. Analogous results also hold for the prime-counting functions
\begin{equation*}
    \psi(x)=\sum_{p^m\leq x}\log p\quad\text{and}\quad\Pi(x)=\sum_{p^m\leq x}\frac{1}{m}.
\end{equation*}

\begin{theorem}\label{tsquaredthm}
    Unconditionally, we have
    \begin{align*}
        &\int_{2}^x\frac{\pi(t)-\li(t)}{t^2}\mathrm{d}t<0,\qquad \int_{2}^x\frac{\theta(t)-t}{t^2}\mathrm{d}t<0,\\
        &\int_{2}^x\frac{\Pi(t)-\li(t)}{t^2}\mathrm{d}t<0,\qquad \int_{2}^x\frac{\psi(t)-t}{t^2}\mathrm{d}t<0
    \end{align*}
    for all $x>2$.
\end{theorem}
Note that since $\pi(t)<\Pi(t)$ and $\theta(t)<\psi(t)$, the apparent negative bias of $\Pi(t)-\li(t)$ and $\psi(t)-t$ is less pronounced than that of $\pi(t)-\li(t)$ or $\theta(t)-t$. In fact, $\int_2^x\Pi(t)-\li(t)\:\mathrm{d}t$ and $\int_2^x\psi(t)-t\:\mathrm{d}t$ change sign infinitely often (see Lemma \ref{signchangelem}). Thus, unlike as in Theorem \ref{tsquaredthm}, Theorems \ref{rhthm1} and \ref{rhthm2} do not hold if $\pi(t)$ and $\theta(t)$ are replaced with $\Pi(t)$ and $\psi(t)$ respectively.

Finally, we show that one cannot do much better than Theorem \ref{tsquaredthm} without further knowledge of the location of the zeros of $\zeta(s)$.

\begin{theorem}\label{tctheorem}
    Let $\omega=\sup\{\Re(s):\zeta(s)=0\}$. If $1/2<\omega\leq 1$ and $c<1+\omega\leq 2$ then,
    \begin{align*}
        \int_2^x\frac{\pi(t)-\li(t)}{t^c}\mathrm{d}t=\Omega_+(1),\qquad\int_2^x\frac{\theta(t)-t}{t^c}\mathrm{dt}=\Omega_+(1),\\ \int_2^x\frac{\Pi(t)-\li(t)}{t^c}\mathrm{d}t=\Omega_+(1),\qquad\int_2^x\frac{\psi(t)-t}{t^c}\mathrm{dt}=\Omega_+(1).   
    \end{align*}
\end{theorem}
Here, as per usual, the notation $g(x)=\Omega_+(1)$ means that there exist arbitrarily large values of $x$ such that $g(x)>0$.

\begin{remark}
    Despite the restrictions in Theorem \ref{tctheorem}, it is conceivable that one may be able to use a slightly (asymptotically) larger weight than $f(t)=1/t^2$ in Theorem \ref{tsquaredthm}. For instance, $f(t)=\log t/t^2$. Such a result would most likely require the use of an explicit zero-free region, e.g.\ \cite[Theorem 1]{M_T_2015} or \cite[Theorem 5]{Ford_2002}. We do not pursue this here.
\end{remark}

\section{Useful lemmas}
In this section, we list a series of useful lemmas. Most of the following results are explicit bounds on prime counting functions which follow directly from existing results in the literature.

\begin{lemma}[{\cite[Theorem 2]{buthe2018analytic}}]\label{buthelem}
    For all $2\leq x\leq 10^{19}$ we have $\pi(x)-\li(x)<0$ and $\theta(x)-x<0$.
\end{lemma}
\begin{lemma}[{\cite[(3.5),(3.6),(3.15),(3.16)]{rosser1962approximate}}]\label{pithetaapprox}
    We have
    \begin{align}
        \frac{x}{\log x}&<\pi(x)<\frac{1.3x}{\log x},&x\geq 17,\label{piapprox}\\
        x-\frac{x}{\log x}&<\theta(x)<x+\frac{0.5x}{\log x},&x\geq 41 \label{thetaapprox}.
    \end{align}
\end{lemma}
\begin{lemma}\label{pithetalemma}
    We have
    \begin{align}
        \Pi(x)-1.9x^{1/2}&<\pi(x)<\Pi(x)-x^{1/2}/\log x,&x\geq 17,\label{piapprox2}\\
        \psi(x)-1.5x^{1/2}&<\theta(x)<\psi(x)-0.98x^{1/2},&x\geq 121\label{thetaapprox2}.
    \end{align}
\end{lemma}
\begin{proof}
    Let $M=\lfloor\frac{\log x}{\log 2}\rfloor$. Then
    \begin{equation*}
        \Pi(x)=\sum_{m=1}^M\frac{1}{m}\pi(x^{1/m})\leq\pi(x)+\frac{M}{2}\pi(x^{1/2})<\pi(x)+\frac{1.3}{\log 2}x^{1/2}<\pi(x)+1.9x^{1/2}
    \end{equation*}
    by Lemma \ref{pithetaapprox}. On the other hand,
    \begin{equation*}
        \Pi(x)=\sum_{m=1}^M\frac{1}{m}\pi(x^{1/m})>\pi(x)+\frac{1}{2}\pi(x^{1/2})>\pi(x)+\frac{x^{1/2}}{\log x}.
    \end{equation*}
    as required. The inequalities in \eqref{thetaapprox2} then follow from equations (3.36) and (3.37) in \cite{rosser1962approximate}.
\end{proof}
\begin{lemma}\label{lilemma}
    For all $x>1$, we have
    \begin{equation*}
        \li(x)<\frac{x}{\log x}+\frac{2x}{\log^2x}.
    \end{equation*}
\end{lemma}
\begin{proof}
    For $x\geq 1865$ the result follows from \cite[Lemma 5.9]{bennett2018explicit}. For smaller values of $x$, the result follows via simple computations.
\end{proof}
\begin{lemma}\label{exactints}
    We have,
    \begin{align*}
        -30000&<\int_2^{3000}\left(\pi(t)-\li(t)\right)\mathrm{d}t<-29000,\\
        -140000&<\int_2^{3000}\left(\theta(t)-t\right)\mathrm{d}t<-130000,\\
        -2900&<\int_2^{3000}\left(\psi(t)-t\right)\mathrm{d}t<-2800.
    \end{align*}
\end{lemma}
\begin{proof}
    Follows by directly computing each integral on Mathematica.
\end{proof}
\begin{lemma}\label{explicitlemma}
    Assuming the Riemann hypothesis, $\left|\int_2^x\left(\psi(t)-t\right)\mathrm{d}t\right|\leq 0.08x^{3/2}$ for all $x\geq 3000$.
\end{lemma}
\begin{proof}
    By \cite[Theorem 27]{ingham1932distribution},
    \begin{equation*}
        \psi_1(x):=\int_{2}^x\psi(t)\:\mathrm{d}t=\frac{x^2}{2}-\sum_{\rho}\frac{x^{\rho+1}}{\rho(\rho+1)}-x\frac{\zeta'}{\zeta}(0)+\frac{\zeta'}{\zeta}(-1)-\sum_{r=1}^\infty\frac{x^{1-2r}}{2r(2r-1)},
    \end{equation*}
    where the sum is taken over all non-trivial zeros $\rho=\beta+i\gamma$ of the Riemann zeta-function. First we note that
    \begin{equation*}
        \left|\sum_{\rho}\frac{x^{\rho+1}}{\rho(\rho+1)}\right|\leq x^{1+\omega}\sum_{\rho}\frac{1}{\gamma(\gamma+1)}\leq 0.04621x^{3/2}
    \end{equation*}
    since $\sum_{\rho}1/\gamma^2=0.046209\ldots$ \cite[Corollary 1]{brent2021accurate}.
    Next, $(\zeta'/\zeta)(0)=\log 2\pi$ \cite[\S 3.8]{Edwards_74} and $(\zeta'/\zeta)(-1)=1.985\ldots$ as computed on Mathematica. Finally,
    \begin{equation*}
        0\leq\sum_{r=1}^\infty\frac{x^{1-2r}}{2r(2r-1)}\leq\frac{x}{2}\sum_{r=1}^\infty\frac{1}{x^{2r}}=\frac{x}{2(x^2-1)}.
    \end{equation*}
    Thus, noting that $\frac{x^2}{2}=\int_2^x t\:\mathrm{d}t+2$,
    \begin{align*}
        \left|\int_2^x\left(\psi(t)-t\right)\mathrm{d}t\right|&\leq x^{3/2}\left(0.04621+\frac{\log(2\pi)}{x^{1/2}}+\frac{2+1.986}{x^{3/2}}+\frac{1}{2x^{1/2}(x^2-1)}\right)\\
        &\leq 0.08x^{3/2}
    \end{align*}
    for all $x\geq 3000$.
\end{proof}

\begin{lemma}[cf. {\cite[pp. 103--104]{ingham1932distribution}}]\label{pililemma}
    Let $Q(x)=\Pi(x)-\li(x)$, $R(x)=\psi(x)-x$ and $R_1(x)=\int_2^xR(t)\:\mathrm{d}t$. Then,
    \begin{equation}
        Q(x)=\frac{R(x)}{\log x}+\frac{R_1(x)}{x\log^2x}+\int_{3000}^x\left(\frac{R_1(t)}{t^2\log^2t}+\frac{2R_1(t)}{t^2\log^3t}\right)\mathrm{d}t+C,
    \end{equation}
    where
    \begin{equation*}
        C=Q(3000)-\frac{R(3000)}{\log(3000)}-\frac{R_1(3000)}{3000\log^2(3000)}=-0.4351\ldots<0.
    \end{equation*}
\end{lemma}
\begin{proof}
    Using integration by parts 
    \begin{equation*}
        \li(x)-\li(3000)=\frac{x}{\log x}+\int_{3000}^{x}\frac{\mathrm{d}t}{\log^2t}-\frac{3000}{\log(3000)}
    \end{equation*}
    so that by partial summation
    \begin{equation*}
        \Pi(x)-\Pi(3000)=\li(x)-\li(3000)+\frac{R(x)}{\log x}-\frac{R(3000)}{\log (3000)}+\int_{3000}^x\frac{R(t)}{t\log^2t}\mathrm{d}t.
    \end{equation*}
    Hence,
    \begin{equation*}
        Q(x)=\frac{R(x)}{\log x}+\int_{3000}^x\frac{R(t)}{t\log^2t}\mathrm{d}t+Q(3000)-\frac{R(3000)}{\log(3000)}.
    \end{equation*}
    A further application of integration by parts then gives the desired result.
\end{proof}
\begin{lemma}\label{signchangelem}
    Let $c\in\mathbb{R}$ and define 
    \begin{equation*}
        \Pi_{1,c}(x):=\int_2^x\frac{\Pi(t)}{t^c}\mathrm{d}t\quad\text{and}\quad\psi_{1,c}(x):=\int_2^x\frac{\psi(t)}{t^c}\mathrm{d}t.
    \end{equation*}
    If $\omega=\sup\{\Re(s):\zeta(s)=0\}$ and $c<1+\omega\leq 2$, then for all $\delta>0$,
    \begin{equation*}
        \Pi_{1,c}(x)-\int_2^x\frac{\li(t)}{t^c}=\Omega_{\pm}(x^{1+\omega-c-\delta})\quad\text{and}\quad\psi_{1,c}(x)-\int_2^x\frac{t}{t^c}\mathrm{d}t=\Omega_{\pm}(x^{1+\omega-c-\delta}).
    \end{equation*}
\end{lemma}
\begin{proof}
    We begin with the integral expression \cite[(18), p. 18]{ingham1932distribution}
    \begin{equation*}
        \log\zeta(s)=s\int_1^\infty\frac{\Pi(x)}{x^{s+1}}\mathrm{d}x,\quad\Re(s)>1.
    \end{equation*}
    Using integration by parts 
    \begin{align*}
        \int_1^\infty\frac{\Pi(x)}{x^{s+1}}\mathrm{d}x&=\int_1^\infty\frac{d\Pi_{1,c}(x)}{x^{s+1-c}}\\
        &=\left[\frac{\Pi_{1,c}(x)}{x^{s+1-c}}\right]_1^\infty+(s+1-c)\int_1^\infty\frac{\Pi_{1,c}(x)}{x^{s+2-c}}\mathrm{d}x\\
        &=(s+1-c)\int_1^\infty\frac{\Pi_{1,c}(x)}{x^{s+2-c}}\mathrm{d}x,
    \end{align*}
    noting that since $\Pi(x)=O(x)$,
    \begin{equation*}
        \left|\frac{\Pi_{1,c}(x)}{x^{s+1-c}}\right|=O\left(x^{c-1-\Re(s)}\int_2^xt^{1-c}\mathrm{d}t\right)=O(x^{1-\Re(s)})=o(1).
    \end{equation*}
    Hence,
    \begin{equation}\label{piintegral}
        \log\zeta(s)=s(s+1-c)\int_1^\infty\frac{\Pi_{1,c}(x)}{x^{s+2-c}}\mathrm{d}x,\quad\Re(s)>1.
    \end{equation}
    Using an analogous argument with the integral expression \cite[(17), p. 18]{ingham1932distribution}
    \begin{equation*}
        -\frac{\zeta'(s)}{\zeta(s)}=s\int_1^\infty\frac{\psi(x)}{x^{s+1}}\mathrm{d}x,\quad\Re(s)>1
    \end{equation*}
    one obtains
    \begin{equation}\label{psiintegral}
        -\frac{\zeta'(s)}{\zeta(s)}=s(s+1-c)\int_1^\infty\frac{\psi_{1,c}(x)}{x^{s+2-c}}\mathrm{d}x,\quad\Re(s)>1.
    \end{equation}
    Equipped with \eqref{piintegral} and \eqref{psiintegral} one can then follow a standard argument (e.g.\ \cite[pp. 90--91]{ingham1932distribution}, \cite[p. 80]{Broughan_17}) \textit{mutatis mutandis} to obtain the desired result.
\end{proof}

\section{Proof of Theorems \ref{rhthm1} and \ref{rhthm2}}
We begin with the case where the Riemann hypothesis is false. By Lemma \ref{signchangelem} with $c=0$, there are arbitrarily large values of $x$ such that $\int_{2}^x\Pi(t)-\li(t)\:\mathrm{d}t>Kx^{\kappa}$ for some positive constants $K>0$ and $\kappa>3/2$. For such values of $x$, we then have by Lemma \ref{pithetalemma} that
\begin{align*}
    \int_2^x\left(\pi(t)-\li(t)\right)\mathrm{d}t&=\int_{2}^{17}\left(\pi(t)-\li(t)\right)\mathrm{d}t+\int_{17}^{x}\left(\pi(t)-\li(t)\right)\mathrm{d}t\\
    &>-\int_{17}^x1.9t^{1/2}\:\mathrm{d}t+\int_{17}^x\left(\Pi(t)-\li(t)\right)\mathrm{d}t+O(1)\\
    &=Kx^\kappa+O(x^{3/2}).
\end{align*}
Thus, there are arbitrarily large values of $x$ such that
\begin{equation*}
    \int_2^x\left(\pi(t)-\li(t)\right)\mathrm{d}t>0
\end{equation*}
as required. The same reasoning holds for the integral over $\theta(t)-t$ using the corresponding bounds for $\theta(t)$ in Lemmas \ref{pithetalemma} and \ref{signchangelem}.

Now, suppose the Riemann hypothesis is true. To show \eqref{piineq} and \eqref{thetaineq} it suffices to consider $x>10^{19}$ in light of Lemma \ref{buthelem}. We begin with the integral over $\theta(t)-t$. By Lemmas \ref{pithetalemma}, \ref{exactints} and \ref{explicitlemma} we have
\begin{align*}
    \int_{2}^x\left(\theta(t)-t\right)\mathrm{d}t&=\int_{2}^{3000}\left(\theta(t)-t\right)\mathrm{d}t+\int_{3000}^x\left(\theta(t)-t\right)\mathrm{d}t\\
    &<-130000-\int_{3000}^x 0.98t^{1/2}\:\mathrm{d}t+\int_{3000}^x\left(\psi(t)-t\right)\mathrm{d}t\\
    &<-130000-\frac{1.96}{3}(x^{3/2}-(3000)^{3/2})+0.08x^{3/2}-\int_{2}^{3000}\left(\psi(t)-t\right)\mathrm{d}t\\
    &<-19000-0.57x^{3/2}<0,
\end{align*}
as required.
    
The integral over $\pi(t)-\li(t)$ requires more work. First we apply Lemmas \ref{pithetalemma}, \ref{exactints} and \ref{pililemma} to obtain
\begin{align}\label{mainineq}
    \int_2^x\left(\pi(t)-\li(t)\right)\mathrm{d}t&<\int_2^{3000}\left(\pi(t)-\li(t)\right)\mathrm{d}t+\int_{3000}^x\left(\pi(t)-\li(t)\right)\mathrm{d}t\notag\\
    &<-29000-\int_{3000}^x\frac{t^{1/2}}{\log t}\mathrm{d}t+\int_{3000}^x\left(\Pi(t)-\li(t)\right)\mathrm{d}t\notag\\
    &<-29000-\int_{3000}^x\frac{t^{1/2}}{\log t}\mathrm{d}t+\int_{3000}^x\left(\frac{R(t)}{\log t}+\frac{R_1(t)}{t\log^2 t}\right.\notag\\
    &\qquad\qquad\qquad+\left.\int_{3000}^t\left(\frac{R_1(u)}{u^2\log^2u}+\frac{2R_1(u)}{u^2\log^3u}\mathrm{d}u\right)\right)\mathrm{d}t,
\end{align}
using the notation from Lemma \ref{pililemma}. Now, by integration by parts and Lemmas \ref{exactints} and \ref{explicitlemma},
\begin{align}\label{partseq}
    \int_{3000}^x\left(\frac{R(t)}{\log t}+\frac{R_1(t)}{t\log^2t}\right)\mathrm{d}t&=\frac{R_1(x)}{\log x}-\frac{R_1(3000)}{\log(3000)}+\int_{3000}^x\frac{2R_1(t)}{t\log^2t}\mathrm{d}t\notag\\
    &\leq 0.08 \frac{x^{3/2}}{\log x}+370+0.16\int_{3000}^x\frac{t^{1/2}}{\log^2t}\mathrm{d}t.
\end{align}
Next,
\begin{align*}
    \int_{3000}^t\left(\frac{R_1(u)}{u^2\log^2u}+\frac{2R_1(u)}{u^2\log^3u}\right)\mathrm{d}u&\leq 0.08\int_{3000}^t\left(\frac{1}{u^{1/2}\log^2u}+\frac{2}{u^{1/2}\log^3u}\right)\mathrm{d}u\\
    &\leq 0.08\int_{3000}^t\frac{3}{u^{1/2}\log^2u}\mathrm{d}u.
\end{align*}
Since $u^{1/4}/\log^2u$ is increasing for $u\geq 3000$ we then have
\begin{equation*}
    0.08\int_{3000}^t\frac{3}{u^{1/2}\log^2u}\mathrm{d}u\leq 0.24\frac{t^{1/4}}{\log^2t}\int_{3000}^t\frac{1}{u^{3/4}}\mathrm{d}u\leq 0.96\frac{t^{1/2}}{\log^2t}.
\end{equation*}
Thus,
\begin{equation}\label{inteq3}
    \int_{3000}^t\left(\frac{R_1(u)}{u^2\log^2u}+\frac{2R_1(u)}{u^2\log^3u}\right)\mathrm{d}u\leq 0.96\frac{t^{1/2}}{\log^2 t}.
\end{equation}
Substituting \eqref{partseq} and \eqref{inteq3} into \eqref{mainineq} then gives
\begin{align}\label{finaleq}
    \int_{2}^x\left(\pi(t)-\li(t)\right)\mathrm{d}t&<-28630-\int_{3000}^x\frac{t^{1/2}}{\log t}\mathrm{d}t+0.08\frac{x^{3/2}}{\log x}+1.12\int_{3000}^x\frac{t^{1/2}}{\log^2 t}\mathrm{d}t.
\end{align}
Now, using integration by parts
\begin{equation}\label{geqint}
    \int_{3000}^x\frac{t^{1/2}}{\log t}\mathrm{d}t=\frac{2}{3}\frac{x^{3/2}}{\log x}-\frac{2}{3}\frac{(3000)^{3/2}}{\log (3000)}+\frac{2}{3}\int_{3000}^x\frac{t^{1/2}}{\log^2t}\mathrm{d}t\geq\frac{2}{3}\frac{x^{3/2}}{\log x}-14000.
\end{equation}
Moreover,
\begin{equation}\label{leqint}
    1.12\int_{3000}^x\frac{t^{1/2}}{\log^2t}\mathrm{d}t\leq 1.12\frac{x^{1/4}}{\log^2x}\int_{3000}^xt^{1/4}\:\mathrm{d}t\leq 0.9\frac{x^{3/2}}{\log^2x}.
\end{equation}
Applying \eqref{geqint} and \eqref{leqint} we have that \eqref{finaleq} reduces to
\begin{equation*}
    \int_2^x\left(\pi(t)-\li(t)\right)\mathrm{d}t<-14000-0.58\frac{x^{3/2}}{\log x}+0.9\frac{x^{3/2}}{\log^2 x}<0,
\end{equation*}
as required.

\section{Proof of Theorem \ref{tsquaredthm}}
First we note that, via a simple computation, all four integrals in question are negative for $2<x\leq 200$. Thus, we may assume throughout that $x>200$.

First we deal with the integrals involving $\pi(t)$ and $\Pi(t)$. By an explicit form of Mertens' theorem \cite[Equation (3.20)]{rosser1962approximate} we have
\begin{equation}\label{merteneq1}
    \sum_{p\leq x}\frac{1}{p}<\log\log x+B+\frac{1}{\log^2x},
\end{equation}
where $B=0.2614\ldots$. Now, by partial summation and Lemma \ref{pithetaapprox}
\begin{equation}\label{1overpeq}
    \sum_{p\leq x}\frac{1}{p}=\frac{\pi(x)}{x}+\int_2^x\frac{\pi(t)}{t^2}\mathrm{d}t>\frac{1}{\log x}+\int_2^x\frac{\pi(t)}{t^2}\mathrm{d}t.
\end{equation}
Moreover, by integration by parts and Lemma \ref{lilemma}
\begin{align}\label{loglogeq}
    \log\log x=\int_e^x\frac{1}{t\log t}\mathrm{d}t&=\frac{\li(x)}{x}-\frac{\li(e)}{e}+\int_e^x\frac{\li(t)}{t^2}\mathrm{d}t\notag\\
    &\leq\frac{1}{\log x}+\frac{2}{\log^2x}-\frac{\li(e)}{e}+\int_2^x\frac{\li(t)}{t^2}\mathrm{d}t-\int_2^e\frac{\li(t)}{t^2}\mathrm{d}t.
\end{align}
Substituting \eqref{1overpeq} and \eqref{loglogeq} into \eqref{merteneq1} gives
\begin{equation}\label{piineq2}
    \int_2^x\frac{\pi(t)-\li(t)}{t^2}\mathrm{d}t<\frac{2.5}{\log^2x}+B-\frac{\li(e)}{e}-\int_2^e\frac{\li(t)}{t^2}\mathrm{d}t<\frac{3}{\log^2x}-0.62<0
\end{equation}
as desired. For the integral involving $\Pi(t)$ we then note that by Lemma \ref{pithetalemma}
\begin{align*}
    \int_2^x\frac{\pi(t)-\li(t)}{t^2}\mathrm{d}t&=\int_2^{200}\frac{\pi(t)-\li(t)}{t^2}\mathrm{d}t+\int_{200}^x\frac{\pi(t)-\li(t)}{t^2}\mathrm{d}t\\
    &>-0.59+\int_{200}^x\frac{\Pi(t)-1.9t^{1/2}-\li(t)}{t^2}\mathrm{d}t\\
    &=-0.59+\left(\frac{3.8}{\sqrt{x}}-\frac{3.8}{\sqrt{200}}\right)+\int_2^{x}\frac{\Pi(t)-\li(t)}{t^2}\mathrm{d}t\\
    &\qquad\qquad\qquad\qquad\qquad\qquad\qquad-\int_{2}^{200}\frac{\Pi(t)-\li(t)}{t^2}\mathrm{d}t\\
    &>\int_{2}^x\frac{\Pi(t)-\li(t)}{t^2}\mathrm{d}t-0.56+\frac{3.8}{\sqrt{x}}
\end{align*}
Substituting this into \eqref{piineq2} then gives
\begin{equation*}
    \int_{2}^x\frac{\Pi(t)-\li(t)}{t^2}\mathrm{d}t<-0.06-\frac{3.8}{\sqrt{x}}+\frac{3}{\log^2x}<0.
\end{equation*}
    
We argue similarly for the integrals involving $\theta(t)$ and $\psi(t)$. In particular, we have \cite[Equation (3.23)]{rosser1962approximate}
\begin{equation*}
    \sum_{p\leq x}\frac{\log p}{p}<\log x+E+\frac{1}{\log x},
\end{equation*}
where $E=-1.332\ldots$. Then, by Lemma \ref{pithetaapprox}
\begin{equation*}
    \sum_{p\leq x}\frac{\log p}{p}=\frac{\theta(x)}{x}+\int_2^x\frac{\theta(t)}{t^2}\mathrm{d}t>1-\frac{1}{\log x}+\int_2^x\frac{\theta(t)}{t^2}\mathrm{d}t
\end{equation*}
and,
\begin{equation*}
    \log(x)=\int_1^x\frac{1}{t}\mathrm{d}t=\int_2^x\frac{t}{t^2}\mathrm{d}t+\log 2.
\end{equation*}
Therefore,
\begin{equation}\label{thetaineq2}
    \int_2^x\frac{\theta(t)-t}{t^2}\mathrm{d}t<E-1+\log 2+\frac{1.5}{\log x}<-1.63+\frac{2}{\log x}<0
\end{equation}
as required. The corresponding result for $\psi(t)$ then follows similar to before by applying Lemma \ref{pithetalemma} to \eqref{thetaineq2}. In particular, we have that
\begin{equation*}
    \int_2^x\frac{\psi(t)-t}{t^2}\mathrm{d}t<-0.83-\frac{3}{\sqrt{x}}+\frac{2}{\log x}<0.
\end{equation*}

\section{Proof of Theorem \ref{tctheorem}}

The result for the integrals involving $\Pi(t)$ and $\psi(t)$ follows immediately from Lemma \ref{signchangelem}. For the integral involving $\pi(t)$ we first note that by Lemma \ref{signchangelem}, for any choice of $\delta>0$ there exists arbitrarily large values of $x$ such that 
\begin{equation*}
    \int_{2}^x\frac{\Pi(t)-\li(t)}{t^c}\:\mathrm{d}t>Kx^{1+\omega-c-\delta}
\end{equation*} 
for some positive constant $K>0$. For such values of $x$, we then have by Lemma \ref{pithetalemma} that
\begin{align}\label{signchangeeq}
    \int_2^x\frac{\pi(t)-\li(t)}{t^c}\:\mathrm{d}t&=\int_{2}^{17}\frac{\pi(t)-\li(t)}{t^c}\:\mathrm{d}t+\int_{17}^{x}\frac{\pi(t)-\li(t)}{t^c}\:\mathrm{d}t\notag\\
    &>-\int_{17}^x1.9t^{1/2-c}\:\mathrm{d}t+\int_{17}^x\frac{\Pi(t)-\li(t)}{t^c}\:\mathrm{d}t+O(1)\notag\\
    &>-\int_{17}^x 1.9t^{1/2-c}\:\mathrm{d}t+Kx^{1+\omega-c-\delta}+O(1).
\end{align}
The integral in \eqref{signchangeeq} satisfies
\begin{equation*}
    \int_{17}^x1.9t^{1/2-c}\:\mathrm{d}t=
    \begin{cases}
        O(x^{3/2-c}),&c<3/2,\\
        O(\log x),&c=3/2,\\
        O(1),&3/2<c<1+\omega.
    \end{cases}
\end{equation*}
Thus, if we take any choice of $\delta<\omega-1/2$ when $c\leq 3/2$, and $\delta<1+\omega-c$ when $3/2<c<1+\omega$, we have
\begin{equation*}
    \int_2^x\frac{\pi(t)-\li(t)}{t^c}\mathrm{d}t>0
\end{equation*}
for arbitrarily large values of $x$ as required. The same reasoning holds for the integral over $(\theta(t)-t)/t^c$ using the corresponding bounds for $\theta(t)$ in Lemmas \ref{pithetalemma} and \ref{signchangelem}.

\section{Discussion and further work}
The general idea in this paper was to consider averaged versions of arithmetic functions in order to gain insight into biases occuring in number theory. The functions $\pi(t)-\li(t)$ and $\theta(t)-t$ in particular exhibit an apparent negative bias and our results reflect this.

There are many other biases occuring in number theory and it would be interesting to consider averaged versions of these. For example, we have:
\begin{enumerate}[label=(\alph*)]
    \item The bias in Mertens' theorems \cite{buthe2015first}, \cite{lamzouri2016bias}.
    \item The Chebyshev bias for primes in arithmetic progressions \cite{rubinstein1994chebyshev}, \cite{granville2006prime}.
    \item The bias of $\lambda(n)$ and related functions \cite{humphries2013distribution}, \cite{martin2021fake}.
\end{enumerate}

One could also attempt to extend our results to more general number fields. In this direction, it is worth noting that Garcia and Lee \cite{garcia2021unconditional} recently proved explicit versions of Mertens' theorems for number fields. Using Garcia and Lee's results could thus allow one to generalise Theorem \ref{tsquaredthm}, whose proof followed directly from an explicit version of Mertens' theorems in the standard setting.

\section{Acknowledgements}
Thanks to my supervisor Tim Trudgian and colleagues Ethan, Michaela and Shehzad for the fun and fruitful discussions on this project. 

\printbibliography

@article{bays2000new,
  title={A new bound for the smallest $x$ with $\pi (x)> \text{li} (x)$},
  author={Bays, C. and Hudson, R. H.},
  journal={Math. Comp.},
  pages={1285--1296},
  year={2000},
  publisher={JSTOR}
}

@article {bennett2018explicit,
    AUTHOR = {Bennett, M. and Martin, G. and O'Bryant, K. and Rechnitzer, A.},
     TITLE = {Explicit bounds for primes in arithmetic progressions},
   JOURNAL = {Illinois J. Math.},
  FJOURNAL = {Illinois Journal of Mathematics},
    VOLUME = {62},
      YEAR = {2018},
    NUMBER = {1-4},
     PAGES = {427--532},
}

@article{brent2021accurate,
  title={Accurate estimation of sums over zeros of the {R}iemann zeta-function},
  author={Brent, R. and Platt, D. and Trudgian, T.},
  journal={Math. Comp.},
  number={90},
  pages={2923--2935},
  year={2021}
}

@book{Broughan_17,
    AUTHOR = {Broughan, K.},
     TITLE = {Equivalents of the {R}iemann hypothesis. {V}ol. 1},
      NOTE = {Arithmetic equivalents},
 PUBLISHER = {Cambridge University Press, Cambridge},
      YEAR = {2017}
}

@article{buthe2015first,
  title={On the first sign change in Mertens' theorem},
  author={B{\"u}the, J.},
  journal={Acta Arith.},
  volume={171},
  number={2},
  pages={183--195},
  year={2015},
  publisher={Institute of Mathematics Polish Academy of Sciences}
}

@article{buthe2018analytic,
  title={An analytic method for bounding $\psi(x)$},
  author={B{\"u}the, J.},
  journal={Math. Comp.},
  volume={87},
  number={312},
  pages={1991--2009},
  year={2018}
}

@book{Edwards_74,
  title={Riemann's {Z}eta {F}unction},
  author={Edwards, H. M.},
  year={1974},
  publisher={Dover Publications},
  address={New York}
}

@incollection {Ford_2002,
    AUTHOR = {Ford, K.},
     TITLE = {Zero-free regions for the {R}iemann zeta function},
 BOOKTITLE = {Number theory for the millennium, {II} ({U}rbana, {IL}, 2000)},
     PAGES = {25--56},
 PUBLISHER = {A K Peters, Natick, MA},
      YEAR = {2002},
}

@article{lamzouri2016bias,
  title={A bias in Mertens’ product formula},
  author={Lamzouri, Y.},
  journal={Int. J. Number Theory},
  volume={12},
  number={01},
  pages={97--109},
  year={2016},
  publisher={World Scientific}
}

@article{granville2006prime,
  title={Prime number races},
  author={Granville, A. and Martin, G.},
  journal={Amer. Math. Monthly},
  volume={113},
  number={1},
  pages={1--33},
  year={2006},
  publisher={Taylor \& Francis}
}

@article{humphries2013distribution,
  title={The distribution of weighted sums of the Liouville function and P\'olya's conjecture},
  author={Humphries, P.},
  journal={J. Number Theory},
  volume={133},
  number={2},
  pages={545--582},
  year={2013},
  publisher={Elsevier}
}

@article{martin2021fake,
  title={Fake Mu's},
  author={Martin, G. and Mossinghoff, M. J. and Trudgian, T. S.},
  journal={arXiv preprint arXiv:2112.05227},
  year={2021}
}

@article{garcia2021unconditional,
  title={Unconditional explicit Mertens’ theorems for number fields and Dedekind zeta residue bounds},
  author={Garcia, S. R. and Lee, E. S.},
  journal={Ramanujan J.},
  volume={57},
  number={3},
  pages={1169--1191},
  year={2022},
  publisher={Springer}
}

@article{karatsuba2004riemann,
  title={On the Riemann Asymptotic Formula for $\pi(x)$},
  author={Karatsuba, AA},
  journal={Dokl. Math.},
  volume={69},
  number={3},
  pages={423--424},
  year={2004},
}

@book{ingham1932distribution,
  title={The distribution of prime numbers},
  author={Ingham, A. E.},
  year={1932},
  publisher={Cambridge University Press}
}

@article{kaczorowski1985sign,
  title={On sign-changes in the remainder-term of the prime-number formula, I},
  author={Kaczorowski, J.},
  journal={Acta Arith.},
  number={44},
  pages={365--377},
  year={1984}
}

@article{littlewood1914distribution,
  title={Sur la distribution des nombres premiers},
  author={Littlewood, J. E.},
  journal={CR Acad. Sci. Paris},
  volume={158},
  number={1914},
  pages={1869--1872},
  year={1914}
}

@article{M_T_2015,
  title={Nonnegative trigonometric polynomials and a zero-free region for the {R}iemann zeta-function},
  author={Mossinghoff, M. J. and Trudgian, T. S.},
  journal={J. Number Theory},
  volume={157},
  pages={329--349},
  year={2015},
  publisher={Elsevier}
}

@article{pintz1991assertion,
  title={On an assertion of Riemann concerning the distribution of prime numbers},
  author={Pintz, J},
  journal={Acta Math. Hungar.},
  volume={58},
  number={3-4},
  pages={383--387},
  year={1991}
}

@article{platt2016first,
  title={On the first sign change of $\theta(x)-x$.},
  author={Platt, D. J. and Trudgian, T. S.},
  journal={Math. Comp.},
  volume={85},
  number={299},
  pages={1539--1547},
  year={2016}
}

@article{Riemann_1859,
  title={Ueber die {A}nzahl der {P}rimzahlen unter einer gegebenen {G}rosse},
  author={Riemann, B.},
  journal={Ges. Math. Werke und Wissenschaftlicher Nachla{\ss}},
  volume={2},
  pages={145--155},
  year={1859}
}

@article{rosser1962approximate,
  title={Approximate formulas for some functions of prime numbers},
  author={Rosser, J. B. and Schoenfeld, L.},
  journal={Illinois J. Math.},
  volume={6},
  number={1},
  pages={64--94},
  year={1962},
  publisher={Duke University Press}
}

@article{rubinstein1994chebyshev,
  title={Chebyshev's bias},
  author={Rubinstein, M. and Sarnak, P.},
  journal={Exp. Math.},
  volume={3},
  number={3},
  pages={173--197},
  year={1994},
  publisher={Taylor \& Francis}
}

@article{saouter2015still,
  title={A still sharper region where $\pi(x)-\li(x)$ is positive},
  author={Saouter, Y. and Trudgian, T. and Demichel, P.},
  journal={Math. Comp.},
  volume={84},
  number={295},
  pages={2433--2446},
  year={2015}
}

@article{schmidt1903anzahl,
  title={{\"U}ber die Anzahl der Primzahlen unter gegebener Grenze},
  author={Schmidt, E.},
  journal={Math. Ann.},
  volume={57},
  number={2},
  pages={195--204},
  year={1903},
  publisher={Springer}
}

@article{stechkin1996asymptotic,
  title={The asymptotic distribution of prime numbers on the average},
  author={Stechkin, S. B. and Popov, A. Y.},
  journal={Russian Math. Surveys},
  volume={51},
  number={6},
  year={1996},
  publisher={IOP Publishing}
}

\end{document}